\documentclass[12pt,reqno]{amsart}

\usepackage{amsthm,amssymb,amstext,amscd,amsfonts,amsbsy,amsxtra,latexsym,cite}
\usepackage{fullpage}
\usepackage{url}
\usepackage{mathrsfs}
\usepackage{verbatim}

\newcommand{\calH}{\mathcal{H}}

\newcommand{\calW}{\mathcal{W}}
\newcommand{\Prob}{\mathbf{P}}
\renewcommand{\o}{\overline}

\newcommand{\sB}{\mathscr{B}}
\newcommand{\sC}{\mathscr{C}}

\theoremstyle{plain}
\newtheorem{theorem}{Theorem}
\newtheorem*{theorem*}{Theorem}

\newtheorem{proposition}[theorem]{\textbf{Proposition}}
\newtheorem{corollary}[theorem]{\textbf{Corollary}}
\newtheorem*{conjecture*}{Conjecture}

\theoremstyle{definition}

\theoremstyle{remark}

\newtheorem*{remark}{Remark}
\newtheorem*{remarks}{Remarks}

\numberwithin{theorem}{section} \numberwithin{equation}{section}

\begin{document}

\title{A partition identity and the universal mock theta function $g_2$}
\author{Kathrin Bringmann}
\address{Mathematical Institute\\University of
Cologne\\ Weyertal 86-90 \\ 50931 Cologne \\Germany}
\email{kbringma@math.uni-koeln.de}
\author{Jeremy Lovejoy}
\address{CNRS \\ LIAFA \\
Universite Denis Diderot - Paris 7 \\
75205 Paris Cedex 13 \\
FRANCE }
\email{lovejoy@liafa.univ-paris-diderot.fr}
\author{Karl Mahlburg}
\address{Department of Mathematics \\
Louisiana State University \\
Baton Rouge, LA 70803\\ U.S.A.}
\email{mahlburg@math.lsu.edu}
\thanks{The research of the first author was supported by the Alfried Krupp Prize for Young University Teachers of the Krupp Foundation and the research leading to these results has received funding from the European Research Council under the European Union's Seventh Framework Programme (FP/2007-2013) / ERC Grant agreement n. 335220 - AQSER.  The third author was supported by NSF Grant DMS-1201435.}

\begin{abstract}
We prove analytic and combinatorial identities reminiscent of Schur's classical partition theorem. Specifically, we show that certain families of overpartitions whose parts satisfy gap conditions are equinumerous with partitions whose parts satisfy congruence conditions.
Furthermore, if small parts are excluded, the resulting overpartitions are generated by the product of a modular form and Gordon and McIntosh's universal mock theta function.  Finally, we give an interpretation for the universal mock theta function at real arguments in terms of certain conditional probabilities.

\end{abstract}

\maketitle

\section{Introduction and statement of results}
\label{S:Intro}
\subsection{Background and motivation}
This paper is motivated by recent results of the first and third authors \cite{Br-Ma} on partitions related to a classical theorem of Schur.   We begin by recalling Schur's theorem.

By a {\it partition} $\lambda$ of $n$ we mean a non-decreasing sequence of integer parts $1 \leq \lambda_1 \leq \cdots \leq \lambda_k$ that sum to $n$; see \cite{And98} for further background.  Throughout the paper we assume that $d \geq 3$ and $1 \leq r < \frac{d}{2}$. For all $n \geq 0$, let $B_{d,r}(n)$ denote the number of partitions of $n$ into parts congruent to $r$, $d-r$, or $d \pmod{d}$ such that $\lambda_{i+1} - \lambda_i \geq d$ with strict inequality if $d \mid \lambda_{i+1}$.  Let $E_{d,r}(n)$ denote the number of partitions of $n$ into distinct parts that are congruent to $\pm r \pmod{d}$.  Schur's theorem is the following.
\begin{theorem*}[Schur, \cite{Sch26}]
For all $n \geq 0$ we have $B_{d,r}(n) = E_{d,r}(n)$.
\end{theorem*}
\noindent
For more on the history of this theorem, its proofs and its ramifications, see \cite{AG93,And68,And94,Pa,Gle28}.

Denote the generating function for $B_{d,r}(n)$ by
\begin{equation*}
\sB_{d,r}(q) := \sum_{n \geq 0}B_{d,r}(n)q^n.
\end{equation*}
Schur's theorem implies that $\sB_{d,r}(q)$ is a modular function, since
\begin{equation}
\label{Eprod}
\sum_{n \geq 0} E_{d,r}(n)q^n = \left(-q^r,-q^{d-r};q^d\right)_{\infty}.
\end{equation}
Here we have used the usual $q$-series notation,
\begin{equation*}
\left(a_1,a_2,\cdots,a_k;q\right)_n := \prod_{j=0}^{n-1}\left(1-a_1q^j\right)\left(1-a_2q^j\right)\cdots\left(1-a_kq^j\right),
\end{equation*}
valid for $n \in \mathbb{N}_{0} \cup \{\infty\}$.

Now let $C_{d,r}(n)$ denote the number of partitions enumerated by $B_{d,r}(n)$ that also satisfy the additional restriction that the smallest part is larger than $d$.  Denote the generating function for $C_{d,r}(n)$ by
\begin{equation*}
\sC_{d,r}(q) := \sum_{n \geq 0} C_{d,r}(n)q^n.
\end{equation*}
Motivated by an observation of Andrews \cite{And68}, the first and third authors recently showed that $C_{d,r}(q)$ is not a modular form, but the following product of the modular form $\sB_{d,r}(q)$ and a certain specialization of the ``universal" mock theta function,
\begin{equation*}
g_3(x;q) := \sum_{n \geq 0} \frac{q^{n(n+1)}}{(x,q/x;q)_{n+1}}.
\end{equation*}
\begin{theorem*}[see Theorem 1.2 of \cite{Br-Ma}]
We have
\begin{equation}
\label{E:C=g3}
\sC_{d,r}(q) = \sB_{d,r}(q) g_3\left(-q^r; q^d\right).
\end{equation}
\end{theorem*}

The universal mock theta function $g_3(x;q)$ is so-named because Hickerson \cite{Hic88-1,Hic88-2} and Gordon and McIntosh \cite{GM12} have shown that each of the classical \emph{odd} order mock theta functions may be expressed, up to the addition of a modular form, as a specialization of $g_3(x;q)$.   There is a second universal mock theta function,
\begin{equation}
\label{E:g2univ}
g_2(x;q) := \sum_{n \geq 0} \frac{(-q; q)_n q^{\frac{n(n+1)}{2}}}
{(x,  q/x; q)_{n+1}},
\end{equation}
which corresponds to the classical \emph{even} order mock theta functions \cite{GM12}.   It was a search for an analogue of \eqref{E:C=g3} with $g_2(x;q)$ in place of $g_3(x;q)$ that led to what follows.

\subsection{Statement of Results}
An {\it overpartition} $\lambda$ of $n$ is a partition of $n$ in which the final occurrence if an integer may be overlined.
Define the $4 \times 4$ matrix $\o{A}_{d,r}$ by
\begin{equation}
\label{Aprimedef}
\o{A}_{d,r} = \bordermatrix{\text{} & \o{r} & \o{d-r} & \o{d} & d \cr \phantom{-}\o{r} & d & 2r & d+r & r \cr \o{d-r} & 2d-2r & d & 2d -r & d-r \cr \phantom{-}\o{d} & 2d-r & d+r & 2d & d \cr \phantom{-}d & d-r & r & d & 0}.
\end{equation}
The rows and columns are indexed by $\overline{r}$, $\overline{d-r}$, $\o{d}$, and $d$, so that, for example, $\o{A}_{d,r}\left(\o{d},\o{d-r}\right) = d+r$.
We consider overpartitions into parts congruent to $r$, $d-r$, or $d \pmod{d}$, where only multiples of $d$ may appear non-overlined.  For $n \geq 0$, let $\o{B}_{d,r}(n)$ denote the number of such overpartitions $\lambda$ of $n$ where
\begin{enumerate}
\item
The smallest part is $\o{r}$, $\o{d-r}$, $\o{d}$, or $2d$ modulo $2d$;
\item
For $u,v \in \{\o{r},\o{d-r},\o{d},d\}$, if $\lambda_{i+1} \equiv u \pmod{d}$ and $\lambda_i \equiv v \pmod{d}$, then $\lambda_{i+1} - \lambda_i \geq \o{A}_{d,r}(u,v)$;
\item
For $u,v \in \{\o{r},\o{d-r},\o{d},d\}$, if $\lambda_{i+1} \equiv u \pmod{d}$ and $\lambda_i \equiv v \pmod{d}$, then $\lambda_{i+1} - \lambda_i \equiv \o{A}_{d,r}(u,v) \pmod{2d}$.  In words, the \emph{actual} difference between two parts must be congruent modulo $2d$ to the \emph{smallest allowable} difference.
\end{enumerate}
Denote the generating function for $\o{B}_{d,r}(n)$ by
$$
\o{\sB}_{d,r}(q) := \sum_{n \geq 0}\o{B}_{d,r}(n)q^n.
$$
Our first result is that $\o{\sB}_{d,r}(q)$ is a quotient of infinite products that is essentially a modular form of weight $-1/2$.
\begin{theorem} \label{main1}
We have
\begin{equation*}
\o{\sB}_{d,r}(q) = \frac{\left(-q^r,-q^{d-r};q^d\right)_{\infty}}{\left(q^{2d};q^{2d}\right)_{\infty}}.
\end{equation*}
\end{theorem}
An immediate corollary is the following combinatorial identity.
\begin{corollary} \label{cor1}
Let $\o{E}_{d,r}(n)$ denote the number of partitions of $n$ into distinct parts congruent to $\pm r$ $\pmod{d}$ and unrestricted parts divisible by $2d$.  Then
for all $n \geq 0$, we have $\o{B}_{d,r}(n) = \o{E}_{d,r}(n)$.
\end{corollary}
To illustrate this identity, let $d=3$, $r=1$, and $n=15$.   Then $\overline{E}_{3,1}(15)=14$, the relevant partitions being
$$
\begin{gathered}
(1,2,4,8), (1,2,5,7), (1,2,6,6), (1,2,12), (1,4,10), (1,6,8), (1,14), \\ (2,5,8), (2,6,7), (2,13), (4,5,6), (4,11), (5,10), (7,8).
\end{gathered}
$$
The matrix $\o{A}_{3,1}$ is
$$
\bordermatrix{\text{} & \o{1} & \o{2} & \o{3} & 3 \cr \o{1} & 3 & 2 & 4 & 1 \cr \o{2} & 4 & 3 & 5 & 2 \cr \o{3} & 5 & 4 & 6 & 3 \cr 3 & 2 & 1 & 3 & 0},
$$
and we find that $\overline{B}_{3,1}(15) = 14$ as well, the relevant overpartitions being
$$
\begin{gathered}
(\o{1},3,3,3,\o{5}), (\o{1},3,\o{4},\o{7}), (\o{1},3,\o{5},6), (\o{1},3,\o{11}), (\o{1},\o{5},\o{9}), \\ (\o{2},3,3,3,\o{4}), (\o{2},3,\o{4},6), (\o{2},3,\o{10}), (\o{2},\o{4},\o{9}), (\o{2},\o{5},\o{8}), (\o{3},6,6), (\o{3},12), (6,\o{9}), (\o{15}).
\end{gathered}
$$
\begin{remarks}
\hskip1in \newline \noindent {\it 1.}
Note that appealing to overpartitions in the definition of $\o{B}_{d,r}(n)$ is convenient but not strictly necessary. In particular, in an overpartition counted by $\o{B}_{d,r}(n)$, a given multiple of $d$ may occur overlined or non-overlined, but not both. Moreover, if the overlines are omitted, then the conditions defining $\o{B}_{d,r}(n)$ ensure that there is no ambiguity when reading the partition from smallest part to largest part.

\medskip
\noindent {\it 2.}
Corollary \ref{cor1} is reminiscent of Schur's theorem if we observe that in the definition of $B_{d,r}(n)$, requiring $\lambda_{i+1} - \lambda_i \geq d$ with strict inequality if $d \mid \lambda_{i+1}$ is equivalent to requiring that if $\lambda_{i+1} \equiv u \pmod{d}$ and $\lambda_i \equiv v \pmod{d}$, then $\lambda_{i+1} - \lambda_i \geq A_{d,r}(u,v)$, where
\begin{equation*}
A_{d,r} := \bordermatrix{\text{} & r & d-r & d \cr \phantom{-}r & d & d+2r & d+r \cr d-r & 2d-2r & d & 2d -r \cr \phantom{-}d & 2d-r & d+r & 2d}.
\end{equation*}
Indeed, the $3 \times 3$ matrix in the upper-left of $\o{A}_{d,r}$ is $A_{d,r}$ with the $(r,d-r)$ entry replaced by $2r$.

\end{remarks}
Next we define $\o{C}_{d,r}(n)$ to be the number of overpartitions of $n$ satisfying conditions $(ii)$ and $(iii)$ in the definition of $\o{B}_{d,r}(n)$, with condition $(i)$ modified in the following way:  all parts are larger than $d$ and the smallest part is congruent to $d, \overline{d+r}, \overline{2d-r}$ or $\o{2d}$ modulo $2d$.  (Observe that there is no overlap between the overpartitions counted by $\o{B}_{d,r}(n)$ and $\o{C}_{d,r}(n)$, unlike the case of $B_{d,r}(n)$ and $C_{d,r}(n)$.)   Denote the generating function for $\o{C}_{d,r}(n)$ by
\begin{equation}
\o{\sC}_{d,r}(q) := \sum_{n \geq 0} \overline{C}_{d,r}(n)q^n.
\end{equation}
We show that $\o{\sC}_{d,r}(q)$ is essentially the product of $\o{\sB}_{d,r}(q)$ and a specialization of $g_2(x;q)$, as follows.
\begin{theorem} \label{main2}
We have
\begin{equation}
\label{E:C=B}
\o{\sC}_{d,r}(q) = \o{\sB}_{d,r}(q) \cdot \left(1-q^d\right) g_2\left(-q^r; q^d\right).
\end{equation}
\end{theorem}
This means that $(1-q^d) g_2\left(-q^r; q^d\right)$ essentially plays the role of a combinatorial correction factor that describes the difference between the enumeration functions $\o{B}_{d,r}$ and $\o{C}_{d,r}$.

Our final result describes a relationship between $\o{C}_{d,r}(n)$ and events in certain probability spaces with infinite sequences of independent events.  In particular, we find an interpretation in terms of conditional probabilities for the universal mock theta function $g_2(x;q)$ evaluated at real arguments; the precise definitions for the following result are found in Section \ref{S:Prob}.
\begin{theorem}
\label{T:g2Prob}
Suppose that $0 < q < 1$ is real.  There are events $Y$ and $Z$ in a certain probability space such that
\begin{equation*}
\Prob(Y \mid Z) = g_2\left(-q^r; q^d\right).
\end{equation*}
\end{theorem}
\begin{remark}
Since probabilities are between $0$ and $1$, Theorem \ref{T:g2Prob} immediately implies that for real $0 \leq q < 1$ we have the bound
\begin{equation*}
g_2\left(-q^r; q^d\right) < 1.
\end{equation*}
\end{remark}

The remainder of the paper is structured as follows.  In the next section we prove Theorems \ref{main1} and \ref{main2} using combinatorial and analytic techniques from the theory of hypergeometric $q$-series.  In Section \ref{S:Prob} we prove Theorem \ref{T:g2Prob} by describing certain probability spaces with infinite sequences of independent events. We conclude in Section \ref{S:Conc} with a brief discussion of open questions arising from this work.

\section{Generating functions, $q$-difference equations, and identities}
\label{S:Gen}

In this section we prove Theorems \ref{main1} and \ref{main2} by deriving and solving $q$-difference equations satisfied by the generating functions for the relevant overpartitions.

Let $\o{B}_{d,r}(m,n)$ (resp. $\o{C}_{d,r}(m,n)$) denote the number of overpartitions counted by $\o{B}_{d,r}(n)$ (resp. $\o{C}_{d,r}(n)$) having $m$ parts.   Define
\begin{equation*}
\o{f}_{d,r}(x) = \o{f}_{d,r}(x;q) := \sum_{m,n\geq 0} \o{B}_{d,r}(m,n)x^m q^n,
\end{equation*}
and note that we have
\begin{equation*}
\o{f}_{d,r}\left(x q^d\right) = \sum_{m,n\geq 0} \o{C}_{d,r}(m,n)x^m q^n.
\end{equation*}
Our goal is to find hypergeometric $q$-series for the cases
\begin{align}
\label{E:BC=f}
\o{\sB}_{d,r}(q)= \o{f}_{d,r}(1;q) = \sum_{n\geq 0} \o{B}_{d,r}(n)q^n, \\
\o{\sC}_{d,r}(q)= \o{f}_{d,r}\left(q^d;q\right) = \sum_{n\geq 0} \o{C}_{d,r}(n)q^n. \notag
\end{align}
We begin by deriving the following $q$-difference equation.
\begin{proposition}
\label{P:frec}
We have
\begin{align*}
\o{f}_{d,r}(x) = \frac{\left(xq^r+ xq^{d-r}\right)}{\left(1-xq^d\right)}\o{f}_{d,r}\left(xq^d\right) + \frac{\left(1+xq^d\right)}{\left(1-xq^{2d}\right)}\o{f}_{d,r}\left(xq^{2d}\right).
\end{align*}
\end{proposition}
\begin{proof}
Suppose that $\lambda$ is an overpartition counted by $\o{B}_{d,r}(m,n)$ for some $m$ and $n$.  Then by condition $(i)$ in the definition of $\o{B}_{d,r}(n)$, the smallest part $\lambda_1$ is either $\overline{r}, \overline{d-r}, \overline{d}$, or something larger.  We look at the four cases separately.

In the first case, we may remove the part of size $\overline{r}$ and any possible occurrences of $d$.  All parts are now larger than $d$ and so we may subtract $d$ from each part to obtain a new overpartition $\mu$.  We claim that $\mu$ is an overpartition counted by $\o{B}_{d,r}(m-t-1,n-r-(m-1)d)$, where $t$ is the number of occurrences of $d$ in $\lambda$.  To see this, first note that in passing from $\lambda$ to $\mu$ we have not affected conditions $(ii)$ or $(iii)$ in the definition of $\o{B}_{d,r}(n)$.  Indeed, subtracting $d$ from each part does not alter the residue class of a given part modulo $d$ or the difference between two parts modulo $2d$.   Hence we only need to verify that $\mu$ satisfies condition $(i)$.  For this, suppose first that there are no occurrences of $d$ in $\lambda$.  Then adding $r$ to the $\overline{r}$-column of \eqref{Aprimedef} we see that $\lambda_2 \equiv \o{d+r}, \o{2d-r}, \o{2d}$, or $d$  modulo $2d$, and so
\begin{equation*}
\mu_1 = \lambda_2 - d \equiv \o{r},\o{d-r},\o{d}, \text{ or } 2d \pmod{2d},
\end{equation*}
as required.  The argument is similar if $d$ does occur in $\lambda$, as then $\mu_1 = \lambda_j - d$, where $\lambda_j$ is the first part in $\lambda$ that is larger than $d$.  Thus, the overpartitions counted by $\o{B}_{d,r}(m,n)$ with $\lambda_1 = \o{r}$ are generated by
\begin{equation*}
\frac{xq^{r}}{1-xq^d}\o{f}_{d,r}\left(xq^d\right).
\end{equation*}

Reasoning along the same lines we find that the overpartitions counted by $\o{B}_{d,r}(m,n)$ with $\lambda_1 = \o{d-r}$ are generated by
\begin{equation*}
\frac{xq^{d-r}}{1-xq^d}\o{f}_{d,r}\left(xq^d\right),
\end{equation*}
the overpartitions counted by $\o{B}_{d,r}(m,n)$ with $\lambda_1=\o{d}$ are generated by
\begin{equation}
\label{E:l1=od}
\frac{xq^{d}}{1-xq^{2d}} \o{f}_{d,r}\left(xq^{2d}\right),
\end{equation}
and the overpartitions counted by $\o{B}_{d,r}(m,n)$ with $\lambda_1 > d $ are generated by
\begin{equation}
\label{E:l1>d}
\frac{1}{1-xq^{2d}} \o{f}_{d,r}\left(xq^{2d}\right).
\end{equation}
For \eqref{E:l1=od} and \eqref{E:l1>d}, note that there is a possibility of nonoverlined parts of size $2d$, but that all subsequent parts are larger than $2d$.
Putting the four cases together gives the statement of the proposition.

\end{proof}

In order to find a hypergeometric solution to the recurrence in Proposition \ref{P:frec}, we introduce an auxiliary function with an additional parameter.

\begin{proposition}
\label{P:Fqdiff}
Suppose that $F(x,y;q)$ satisfies the $q$-difference equation
\begin{equation*}
F(x,y;q) = \frac{\left(xy + xy^{-1} q\right)}{1-xq} F(xq, y; q) + \frac{1 + xq}{1-xq^2} F\left(xq^2, y; q\right)
\end{equation*}
for all complex parameters with $|x|, |q| < 1$, and $F(x,y; q) \to 1$ as $x \to 0$.  Then
\begin{equation*}
F(x,y;q) =
\frac{(-x;q)_{\infty}}{(xq;q)_{\infty}}\sum_{n \geq 0} \frac{(y,y^{-1}q; q)_n(-x)^n}{(q^2;q^2)_n}
\frac{(-xy, -xy^{-1}q; q)_\infty}{(xq, -q; q)_\infty}
\sum_{n \geq 0} \frac{(-x,x; q)_n q^{\frac{n(n+1)}{2}}}{(q,-xy, -xy^{-1}q; q)_n}.
\end{equation*}
\end{proposition}

Recalling Proposition \ref{P:frec} and plugging in $q \mapsto q^d, y = q^r$ and $x = 1$ or $x = q^d$ to Proposition \ref{P:Fqdiff}, we immediately obtain the following formulas, which are Theorems \ref{main1} and \ref{main2} (also recall \eqref{E:g2univ}).
\begin{corollary}
\label{C:f1q}
We have
\begin{align*}
\o{\sB}_{d,r}(q) &= \frac{\left(-q^r, -q^{d-r}; q^d\right)_\infty}{\left(q^{2d}; q^{2d}\right)_\infty}, \\
\o{\sC}_{d,r}(q) &= \left(1-q^d\right) \frac{\left(-q^r, -q^{d-r}; q^d\right)_\infty}{\left(q^{2d}; q^{2d}\right)_\infty}
\sum_{n \geq 0} \frac{\left(-q^d; q^d\right)_n q^{\frac{dn(n+1)}{2}}}{\left(-q^r, -q^{d-r}; q^d\right)_{n+1}} \\
&= \left(1-q^d\right) \frac{\left(-q^r, -q^{d-r}; q^d\right)_\infty}{\left(q^{2d}; q^{2d}\right)_\infty} g_2\left(-q^r,q^d\right).
\end{align*}
\end{corollary}

We may also plug in $x=-1$ or $x=-q^d$ to obtain formulas for $\o{B}_{d,r,+}(n) - \o{B}_{d,r,-}(n)$ and $\o{C}_{d,r,+}(n) - \o{C}_{d,r,-}(n)$, where $\o{B}_{d,r,\pm}(n)$ (resp. $\o{C}_{d,r,\pm}(n))$ is the number of overpartitions of $n$ counted by $\o{B}_{d,r}(n)$ (resp. $\o{C}_{d,r}(n)$) having an even/odd number of parts.
\begin{corollary}
We have
\begin{equation*}
\begin{aligned}
\sum_{n \geq 0} \left(\o{B}_{d,r,+}(n) - \o{B}_{d,r,-}(n)\right)q^n &= \frac{\left(q^r, q^{d-r}; q^d\right)_\infty}{\left(-q^{d}; q^{d}\right)_{\infty}^2}, \\
\sum_{n \geq 0} \left(\o{C}_{d,r,+}(n) - \o{C}_{d,r,-}(n)\right)q^n
&= \left(1 + q^d\right)\frac{\left(q^r, q^{d-r}; q^d\right)_\infty}{\left(-q^{d}; q^{d}\right)_{\infty}^2}
\sum_{n \geq 0} \frac{\left(-q^d; q^d\right)_n q^{\frac{dn(n+1)}{2}}}{\left(q^r, q^{d-r}; q^d\right)_{n+1}}\\
&= \left(1 + q^d\right)\frac{\left(q^r, q^{d-r}; q^d\right)_\infty}{\left(-q^{d}; q^{d}\right)_{\infty}^2}
g_2\left(q^r; q^d\right).
\end{aligned}
\end{equation*}
\end{corollary}

\begin{proof}[Proof of Proposition \ref{P:Fqdiff}]
We first find a hypergeometric solution to the $q$-difference equation (which must be unique as in Lemma 1 of \cite{And68Q}), and then use a $_3 \phi_2$ transformation to obtain the result.  It is convenient to renormalize the $q$-difference equation by defining
\begin{equation}
\label{E:Gdr}
G(x,y;q) = G(x) := \frac{(xq; q)_\infty}{(-x; q)_\infty} F(x,y;q).
\end{equation}
One then easily sees that this function satisfies the equation
\begin{equation}
\label{E:Gqdiff}
(1+x) G(x) = \left(xy + xy^{-1}q\right) G(xq) + (1-xq) G\left(xq^2\right).
\end{equation}
If we expand $G$ as a series in $x$, writing $G(x) =: \sum_{n \geq 0} A_n x^n$, then isolating the $x^n$ coefficient of \eqref{E:Gqdiff} implies (after some simplification) that
\begin{equation*}
A_n = -\frac{\left(1 - yq^{n-1}\right)\left(1- y^{-1}q^n\right)}{1-q^{2n}}A_{n-1}.
\end{equation*}
Since we clearly have $A_0 = 1$, we find a hypergeometric series for $G(x,y;q)$, which combines with \eqref{E:Gdr} to give the solution
\begin{equation*}
F(x,y;q) = \frac{(-x;q)_{\infty}}{(xq;q)_{\infty}}
\sum_{n \geq 0} \frac{(y,y^{-1}q; q)_n(-x)^n}{(q^2;q^2)_n}.
\end{equation*}
Finally, we use the following $_{3}\phi_{2}$ transformation (which is found in an equivalent form as equation (III.10) in \cite{GR90})
\begin{equation*}
\sum_{n\geq 0}\frac{\left(\frac{aq}{bc},d,e;q\right)_{n}}{\left(q,\frac{aq}{b},\frac{aq}{c};q\right)_{n}}\left(\frac{aq}{de}\right)^n= \frac{\left(\frac{aq}{d},\frac{aq}{e},\frac{aq}{bc};q\right)_{\infty}}{\left(\frac{aq}{b},\frac{aq}{c},\frac{aq}{de};q\right)_{\infty}}
\sum_{n\geq 0}\frac{\left(\frac{aq}{de},b,c;q\right)_{n}}{\left(q,\frac{aq}{d},\frac{aq}{e};q\right)_{n}}\left(\frac{aq}{bc}\right)^n.
\end{equation*}
Setting $a=-x$, $b=x$, $c \to \infty$, $d=y$, and $e= y^{-1} q$ gives the result.
\end{proof}

\section{Probabilistic interpretation of universal mock theta function}
\label{S:Prob}

In this section we prove the remarkable fact that Gordon and McIntosh's universal mock theta function at real arguments occurs naturally as the conditional probability of certain events in simple probability spaces.

For $k \geq 1$, define independent events $N_{kd}$ and $O_k$ that occur with probabilities
\begin{equation}
\label{E:NOProb}
\Prob(N_{kd}) = n_{kd} := q^{kd} \qquad \text{and} \qquad \Prob(O_k) = o_k := \frac{q^k}{1+q^k},
\end{equation}
with complementary probabilities $\o{n}_{kd} := 1 - n_{kd}, \o{o}_k := 1 - o_k.$  We further let $T_k$ denote trivial events that each occur with probability $1$. For any events $R$ and $S$, we adopt the space-saving notational conventions $RS := R \cap S.$

We now define additional events based on the sequences of $N_{kd}$s and $O_k$s.  First we introduce further auxiliary events, as for $j \geq 1$ we set
\begin{equation*}
E_j := \begin{cases}
O_{nd+r} \o{O}_{nd+d-r} \o{O}_{(n+1)d}  \cup \o{O}_{nd+r} \qquad &\text{if } j = nd + r, \\
O_{nd+d-r} \o{O}_{(n+1)d} \cup \o{O}_{nd +d-r} &\text{if } j = nd + d-r, \\
O_{(n+1)d} \o{O}_{(n+1)d+r} \o{O}_{(n+1)d+d-r} \o{O}_{(n+2)d} \o{N}_{(n+1)d} \cup \o{O}_{(n+1)d} & \text{if } j = (n+1)d, \\
T_j &\text{if } j \not \equiv 0, \pm r \pmod{d}.
\end{cases}
\end{equation*}
Note that $E_j$ is independent from all $N_{kd}$ and $O_k$s if $j \not \equiv 0, \pm r \pmod{d}.$  Our main focus in this section is then on the events
\begin{equation*}
W_{d,r} := \bigcap_{j \geq 1} E_j, \qquad X_{d,r} := \bigcap_{j \geq d+1} E_j.
\end{equation*}
In words, $W_{d,r}$ is the event such that if $O_{dn+r}$ occurs, then $O_{dn+d-r}$ and $O_{(n+1)d}$ do not occur; if $O_{dn+d-r}$ occurs, then $O_{(n+1)d}$ does not occur; and if $O_{(n+1)d}$ occurs, then $N_{(n+1)d},\ O_{(n+1)d+r}$,\\ $O_{(n+1)d+d-r}$ and $O_{(n+2)d}$ do not occur. The event $X_{d,r}$ has the same conditions beginning from $O_{d+r}$, with no restrictions on $N_d, O_r, O_{d-r}$ or $O_d$.

Note that by using basic set operations, the conditions for the event $W_{d,r}$ can alternatively be written as
\begin{align}
\label{E:Wdisjoint}
\bigcap_{n \geq 0} \Big(O_{nd+r} & \o{O}_{nd+d-r} \o{O}_{(n+1)d} \cup
 \o{O}_{nd+r} O_{nd+d-r} \o{O}_{(n+1)d} \\
 & \cup
\o{O}_{nd+r} \o{O}_{nd+d-r} \o{N}_{(n+1)d} \o{O}_{(n+1)d+r} \o{O}_{(n+1)d+d-r} \o{O}_{(n+2)d}\Big) \notag
\end{align}
In other words, either exactly one of $O_{nd+r}$ or $O_{nd+d-r}$ occurs, or neither of them do, with resulting gap conditions on subsequent events.

\begin{theorem}
\label{T:prob}
Suppose that $0 < q < 1$.  The following identities hold:
\begin{enumerate}
\item
\label{T:prob:W|X}
$\displaystyle \Prob(W_{d,r} \mid X_{d,r})
= \frac{1}{\left(1+q^r\right) \left(1+q^{d-r}\right) \left(1+q^d\right)} \cdot \frac{1}{g_2\left(-q^r; q^d\right)}$,
\item
\label{T:prob:F|W}
$\displaystyle \Prob(\o{O}_r \o{O}_{d-r} \o{O}_d \mid W_{d,r}) = g_2\left(-q^r; q^d\right).$
\end{enumerate}
\end{theorem}
\begin{proof}
For fixed $(d,r)$, let $\calW_k$ denote the event that all of the conditions in the definition of $W_{d,r}$ are met beginning from $E_{kd+r}$ (or, equivalently, from $E_{kd+1}$), so that
\begin{equation*}
\calW_k = \bigcap_{j > kd} E_j.
\end{equation*}
For example, $\calW_0 = W_{d,r}$, and $\calW_{1} = X_{d,r}.$

Then it is clear from \eqref{E:Wdisjoint} that the probabilities of these events satisfy the recurrence
\begin{align}
\label{E:Wkrec}
\Prob(\calW_{k}) = & \left(o_{kd+r} \o{o}_{kd+d-r} \o{o}_{(k+1)d} + \o{o}_{kd+r} o_{kd+d-r} \o{o}_{(k+1)d} \right) \Prob(\calW_{k+1}) \\
& \qquad + \o{n}_{(k+1)d} \cdot \o{o}_{kd+r} \o{o}_{kd+d-r} \o{o}_{(k+1)d+r} \o{o}_{(k+1)d+d-r} \o{o}_{(k+2)d} \Prob(\calW_{k+2}). \notag
\end{align}

In order to compare these probabilities to overpartitions counted by $\o{B}_{d,r}(n)$, we define yet another renormalization of the generating functions.  Specifically, let
\begin{equation}
\label{E:hdr}
h_{d,r}(x) = h_{d,r}(x;q) := \frac{\left(xq^d; q^d\right)_\infty}{\left(-xq^r, -xq^{d-r}, -xq^d; q^d\right)_\infty} \o{f}_{d,r}(x;q),
\end{equation}
so that by Proposition \ref{P:frec} we have the $q$-difference equation
\begin{align}
\label{E:hrec}
h_{d,r}(x) = & \frac{xq^r + xq^{d-r}}{\left(1+xq^r\right)\left(1+xq^{d-r}\right)\left(1+xq^d\right)} h_{d,r}\left(xq^d\right) \\
& + \left(1-xq^d\right) \cdot \frac{1}{\left(1+xq^r\right)\left(1+xq^{d-r}\right)\left(1+xq^{d+r}\right)\left(1+xq^{2d-r}\right)\left(1+xq^{2d}\right)} h_{d,r}\left(xq^{2d}\right). \notag
\end{align}

If we now define $\calH_k = \calH_k(q) := h_{d,r}\left(q^{kd}\right)$ and recall \eqref{E:NOProb}, then \eqref{E:hrec} implies that the recurrence \eqref{E:Wkrec} holds with $\calH_k$ in place of $\Prob(\calW_k).$
We observe that as $k \to \infty$, we have the limit $\calH_k \to 1$, because $h_{d,r}(x) \to 1$ as $x \to 0$.  Similarly, we also have $\Prob(\calW_k) \to 1$ since there are no conditions on any $N_j$ or $O_j$ in the limit.  This boundary condition guarantees that the recurrence has a unique solution, hence
\begin{equation*}
\Prob(\calW_k) = \calH_k(q) = h_{d,r}\left(q^{kd}\right).
\end{equation*}

We can now complete the proof of the theorem.  For part \ref{T:prob:W|X}, we calculate
\begin{equation}
\label{E:W|X}
\Prob(W_{d,r} \mid X_{d,r}) = \frac{\Prob\left(W_{d,r}\right)}{\Prob\left(X_{d,r}\right)} = \frac{\Prob\left(\calW_0\right)}{\Prob\left(\calW_1\right)}
= \frac{\left(1-q^d\right) \o{f}_{d, r}(1)}
{\left(1+q^r\right)\left(1+q^{d-r}\right)\left(1+q^d\right) \o{f}_{d, r}\left(q^d\right)},
\end{equation}
where the last equality is due to \eqref{E:hdr}.
The theorem statement then follows from \eqref{E:C=B} and \eqref{E:BC=f} which together imply that $\o{f}_{d,r}(q^d) = \o{f}_{d,r}(1) \cdot (1-q^d) g_2(-q^r; q^d).$

For part \ref{T:prob:F|W}, we similarly have
\begin{equation*}
\Prob\left(\o{O}_r \o{O}_{d-r} \o{O}_d \mid W_{d,r}\right)
= \frac{\Prob\left(\o{O}_r \o{O}_{d-r} \o{O}_d \cap W_{d,r}\right)}{\Prob\left(W_{d,r}\right)}
= \frac{\Prob\left(\o{O}_r \o{O}_{d-r} \o{O}_d\right) \Prob\left(X_{d,r}\right)}{\Prob\left(W_{d,r}\right)}
= g_2\left(-q^r; q^d\right),
\end{equation*}
where the final equality follows from \eqref{E:NOProb} and the inverse of \eqref{E:W|X}.

\end{proof}

\section{Concluding Remarks}
\label{S:Conc}

It would be interesting to see a bijective proof of Theorem \ref{main1} and/or the fact that
\begin{equation*}
\o{\sB}_{d,r}(q) = \frac{\sB_{d,r}(q)}{(q^{2d}; q^{2d})_\infty},
\end{equation*}
which follows from comparing \eqref{Eprod} and Schur's theorem with Theorem \ref{main1}.  It would also be interesting to see if there are generalizations of Theorem \ref{main1} analogous to generalizations of Schur's theorem in \cite{AnSchur1,AnSchur2}.


\end{document}